\documentclass[12pt,a4paper]{amsart}
\usepackage[bottom]{footmisc}
\usepackage{mathrsfs}
\usepackage{amsmath}
\usepackage{amsthm}
\usepackage{amsfonts}
\usepackage{latexsym}
\usepackage{graphicx}
\usepackage{hyperref}
\usepackage{amssymb}
\usepackage{epsf}
\usepackage{float}
\usepackage{fancyhdr}
\allowdisplaybreaks \makeatletter
\def\rightharpoonfill@{\arrowfill@\relbar\relbar\rightharpoonup}
\DeclareRobustCommand{\overrightharpoon}{\mathpalette{\underarrow@\rightharpoonfill@}}
\makeatother

\allowdisplaybreaks

\begin{document}
\newcommand{\beq}{\begin{equation}}
\newcommand{\eneq}{\end{equation}}
\newtheorem{thm}{Theorem}[section]
\newtheorem{cor}[thm]{Corollary}
\newtheorem{lem}[thm]{Lemma}
\newtheorem{pro}[thm]{Proposition}
\newtheorem{defn}[thm]{Definition}
\newtheorem{rem}[thm]{Remark}
\newtheorem{cl}[thm]{Claim}
\newtheorem{conjecture}{Conjecture}[section]
\newtheorem{corollary}{Corollary}[section]
\newtheorem{definition}{Definition}[section]
\newtheorem{proposition}{Proposition}[section]
\newtheorem{example}{Example}[section]
\newtheorem{lemma}{Lemma}[section]
\newtheorem{claim}{Claim}[section]
\newtheorem{problem}{Problem}[section]
\newtheorem{remark}{Remark}[section]
\newtheorem{theorem}{Theorem}[section]
\title{Well-posedness for the Cauchy problem of spatially weighted dissipative equation}
\author{Ziheng Tu$^{1}$\ \ \ \ Xiaojun Lu$^2$}
\pagestyle{fancy}                   
\lhead{Z. Tu and X. Lu}
\rhead{Spatially weighted dissipative equation} 
\thanks{Corresponding author: Xiaojun Lu, Department of Mathematics \& Jiangsu Key Laboratory of
Engineering Mechanics, Southeast University, 210096, Nanjing, China}
\thanks{Email addresses: tuziheng@zufe.edu.cn(Ziheng Tu), lvxiaojun1119@hotmail.de(Xiaojun Lu)}
\date{}
\maketitle
\begin{center}
1. School of Mathematics and Statistics\\Zhejiang University of
Finance and Economics, 310018, Hangzhou, P.R.China\\
2. Department of Mathematics \& Jiangsu Key Laboratory of
Engineering Mechanics, Southeast University, 210096, Nanjing, China
\end{center}
\begin{abstract}
This paper mainly investigates the Cauchy problem of the spatially
weighted dissipative equation with initial data in the weighted
Lebesgue space. A generalized Hankel Transform is introduced to
derive the analytical solution and a special Young's Inequality has
been applied to prove the space-time estimates for this type of
equation.
\end{abstract}
\section{Introduction}
\indent We consider the Cauchy problem of the following spatially
weighted dissipative equation
\begin{equation}\label{1.1}\left\{\begin{array}{ll}
\partial_t u-|x|^{\beta} \Delta u=\pm|u|^bu\ \ (x,t)\in \mathbb{R}^n\times [0,\infty),\\
u(x,0)=u_0(x)\ \ \ \ \ x \in \mathbb{R}^n.
\end{array}\right.\end{equation}

Nowadays, the diffusion equations with variable coefficients have
wide application in physics, chemistry and engineering etc., and
attract more attention. As far as we know, there is few literature
on the well-posedness and space-time estimates of such type of
dissipative equation. In \cite{Miao01}, Miao studied the general
parabolic type equation where the diffusion operator
$A=\sum_{|\alpha|\leq 2m}a_{\alpha}(x)\partial^\alpha$ is strictly
parabolic type. A unified method based on the space-time estimates
has been introduced to demonstrate the well-posedness result. It is
worth noticing that $|x|^\beta\Delta$ does not absolutely satisfy
the strict parabolic condition. Compared with the standard Heat
equation, or the fractional dissipative equation(see
\cite{Miao08},\cite{Zhai09}), the spatial weight prevents us from
applying the partial Fourier transform. As a result, we need to
explore new approaches.

In this paper, we aim to solve this issue by introducing a special
Hankel transform. In accordance with the coefficient, we call it
$\beta$-Hankel transform. As we know, the standard Hankel transform
is a natural generalization of the Fourier transform of radial
functions. It is closely related with the Bessel operator and has
some nice properties such as $L^2$ isometry and self-adjointness,
etc. \cite{BR97},\cite{Haimo},\cite{Hirschman},\cite{MB95},\cite{PP97}.
Hankel transform has widely applications in the study of PDEs,
especially for the radial solutions of dispersive equations. For
instance, in \cite{FP03}\cite{FP01}  F. Planchon et al. applied the
Hankel transform to obtain the Strichartz estimates for the wave
equation with inverse square potential. In \cite{Tao}, Tao gave the
double end-point Strichartz estimate of the Schr\"{o}dinger
equation. In \cite{CHEN}, Chen studied the similar Schr\"{o}dinger
equation with inverse square potential. Recently, Miao \cite{Miao15}
obtained the maximal estimate to the the Schr\"{o}dinger equation
with inverse square potential. All the above work used the Hankel
transform to get the explicit solution formula.

Before we state the main result, we first introduce the definitions
of admissible and generalized triplets and give the functional space
we use.
\begin{defn}
The triplet $(m,p,q)$ is called an admissible triplet (for the
$k$-th model) if
 $$\frac1m=\frac{n-\beta+2k}{2-\beta}(\frac1q-\frac1p),$$
where
$$1<q\leq p<\left\{\begin{array}{ll}\frac{q(n-\beta+2k)}{n+2k-2}, \ \ \ \mbox{for}\ \ n>2-2k;\\
 \infty,\ \ \ \ \ \ \mbox{for}\ \ n\leq2-2k.\ \ \ \end{array}\right. $$
\end{defn}
\begin{defn}
The triplet $(m,p,q)$ is called a generalized admissible triplet
(for the $k$-th model) if
 $$\frac1m=\frac{n-\beta+2k}{2-\beta}(\frac1q-\frac1p),$$
where
$$1<q\leq p<\left\{\begin{array}{ll}\frac{q(n-\beta+2k)}{n+2k-2q+(q-1)\beta}, \ \ \ \mbox{for}\ \ n>2q+(1-q)\beta-2k;\\
 \infty,\ \ \ \ \ \ \mbox{for}\ \ n\leq2q+(1-q)\beta-2k.\ \ \ \end{array}\right. $$
\end{defn}
Here $k$ is a positive integer associated with the $k$-th model
which will be introduced in section 2.
\begin{rem}
(i) One can easily find that for the given $\beta$ and $k$, $m$ is unique determined by $p$ and $q$. Usually we write $m=m(p, q)$ .\\
(ii) It is easy to see that $q<m\leq\infty$ if $(m, p, q)$ is an admissible triplet. The condition $q<m$ is required
from the application of Marcinkiewicz interpolation theorem in Lemma 3.2.\\
(iii) It is easy to see that $1<m\leq\infty$ if $(m, p, q)$ is a
generalized admissible triplet.
\end{rem}

Now we define the $\mathcal{L}$-type space as
$$X(I):=C(I;L_{d\eta}^q(\mathbb{R}^+))\cap L^m(I;L_{d\eta}^p(\mathbb{R}^+)),$$ and $\mathcal{C}$-type space as
$$Y(I):=C_b(I;L^q_{d\eta}(\mathbb{R}^+))\cap
\mathcal{\dot{C}}_{m}(I;L_{d\eta}^p(\mathbb{R}^+)),$$ where
$I=[0,T)$ for $T>0$. And the weighted Lebesgue space
$L_{d\eta}^p(\mathbb{R}^+)$, time-weighted space-time Banach space
$\mathcal{C}_{\sigma}(I;L_{d\eta}^q(\mathbb{R}^+))$ and the
corresponding homogeneous space
$\mathcal{\dot{C}}_{\sigma}(I;L_{d\eta}^q(\mathbb{R}^+))$ are
defined as follows,
\begin{eqnarray*}
&&L_{d\eta}^p(\mathbb{R}^+):=\left\{f\in \mathcal{S}'(0,\infty)\left|\|f\|^p_{L_{d\eta}^p(\mathbb{R}^+)}=\int_0^\infty |{f}(r)|^p d\eta(r)<\infty\right.\right\};\\
&&\mathcal{C}_{\sigma}(I;L_{d\eta}^q(\mathbb{R}^+)):=\left\{f\in C(I;L_{d\eta}^q(\mathbb{R}^+))\left|\|f;\mathcal{C}_{\sigma}(I;L_{d\eta}^q(\mathbb{R}^+))\|=\sup_{t\in I}t^{\frac1\sigma}\|f\|_{L_{d\eta}^q(\mathbb{R}^+)}<\infty\right.\right\};\\
&&\mathcal{\dot{C}}_{\sigma}(I;L_{d\eta}^q(\mathbb{R}^+)):=\left\{f\in
\mathcal{C}_{\sigma}(I;L_{d\eta}^q(\mathbb{R}^+))\left|\lim_{t\rightarrow0^+}t^{\frac1\sigma}\|f\|_{L_{d\eta}^q(\mathbb{R}^+)}=0\right.\right\}.
\end{eqnarray*}
Define the norm
$$\|\cdot\|_{X(I)}:= \|\cdot\|_{L^\infty(I;L_{d\eta}^q(\mathbb{R}^+))}+\|\cdot\|_{L^m(I;L_{d\eta}^p(\mathbb{R}^+))}$$
and
$$\|\cdot\|_{Y(I)}:= \|\cdot\|_{L^\infty(I;L_{d\eta}^q(\mathbb{R}^+))}+\sup_{t\in I}t^{\frac1m}\|\cdot\|_{L_{d\eta}^p(\mathbb{R}^+)}.$$

At the moment, it is ready for us to introduce the main results. Consider the radial solution $u(t,r)$ of \eqref{1.1} satisfying
\begin{equation}\left\{\begin{array}{ll}
\partial_tu-r^\beta(\partial_{rr}u+\frac{n-1}r\partial_ru)=F(u(t,r)),\\
u(0,r)=u_0(r),
\end{array}\right.\label{radial}
\end{equation}
where the nonhomogeneous term $F(u)=\pm |u|^bu$. Let $k=0$,
$\gamma=\frac{n-\beta}{2-\beta}$ and $d\eta(r)=r^{n-1-\beta}dr$, we
have the following theorem on the existence of local solutions or
global small solutions.
\begin{thm}
 Let $1\leq q_0=\gamma b\leq q$ and $u_0\in L^q_{d\eta}(\mathbb{R}^+)$. Assume $(m,p,q)$ is an arbitrary admissible triplet with $k=0$. \\
(i) There exits $T>0$ and a unique solution $u\in X(I)$ to the problem \eqref{radial}, where $T=T(\|u_0\|_{L^q_{d\eta}})$ depends on $\|u_0\|_{L^q_{d\eta}}$ for $q>q_0$.\\
(ii) If $q=q_0$ then $T=\infty$ provided that
$\|u_0\|_{L^q_{d\eta}}$ is sufficiently small. In other words, there
exits a global small solution
$u\in C_b([0,\infty);L_{d\eta}^q(\mathbb{R}^+))\cap L^m([0,\infty);L_{d\eta}^p(\mathbb{R}^+))$.\\
(iii) Let $I=[0,T^*)$ be the maximal existence interval of the
solution $u$ to the problem \eqref{radial} such that $u\in
C_b([0,T^*);L_{d\eta}^q(\mathbb{R}^+))\cap
L^m([0,T^*);L_{d\eta}^p(\mathbb{R}^+))$ for $q>q_0$. Then,
$$\|u(t)\|_{L^q_{d\eta}}\geq\frac{C}{(T^*-t)^{\frac1b-\frac{\gamma}{q}}}.$$
\end{thm}

In a similar manner, we can also prove the following well-posedness
results under the $\mathcal{C}$-space theory.
\begin{thm}
Let $\gamma=\frac{n-\beta}{2-\beta}$, $1\leq q_0=\gamma b\leq q$ and $u_0\in L^q_{d\eta}(\mathbb{R}^+)$. Assume $(m,p,q)$ is an arbitrary generalized admissible triplet with $k=0$. \\
(i) There exits $T>0$ and a unique mild solution $u\in Y(I)$ to the problem \eqref{radial}, where $T=T(\|u_0\|_{L^q_{d\eta}})$ depends on $\|u_0\|_{L^q_{d\eta}}$ for $q>q_0$.\\
(ii) If $q=q_0$ then $T=\infty$ provided that
$\|u_0\|_{L^q_{d\eta}}$ is sufficiently small. In other words, there
exits a global small solution
$u\in C_b([0,\infty);L_{d\eta}^q(\mathbb{R}^+))\cap \mathcal{\dot{C}}_{m}([0,\infty);L_{d\eta}^p(\mathbb{R}^+))$.\\
(iii) Let $I=[0,T^*)$ be the maximal existence interval of the
solution $u$ to the problem \eqref{radial} such that $u\in
C_b([0,T^*);L_{d\eta}^q(\mathbb{R}^+))\cap
\mathcal{\dot{C}}_{m}([0,T^*);L_{d\eta}^p(\mathbb{R}^+))$ for
$q>q_0$. Then:
$$\|u(t)\|_{L^q_{d\eta}}\geq\frac{C}{(T^*-t)^{\frac1b-\frac{\gamma}{q}}}.$$
\end{thm}

\begin{rem}
Our method can be further applied to the diffusion operator with
inverse square potential such as
$$A=|x|^{\beta}(\Delta+\frac{a}{|x|^2}).$$ As the technique reason,
the coefficient $\beta$ is restricted on $[0,2)$ in this paper. It
is worth noticing the case $\beta=0$ is reduced to the standard Heat
equation.
\end{rem}
The paper is organized as follows. In Section 2, we show some
preliminary work. The theory of Hankel transform and the property of
Bessel function has been revisited. Then, we introduce
$\beta$-Hankel transform and its inverse transform. With these
definitions, a weighted $L^2$ isometry is investigated which is
similar to the standard Hankel transform.  Moreover, in order to
obtain the space-time estimates, an associated convolution operator
and the Young's inequality are introduced. In Section 3, the
semigroup $S_{\mu}(t)=e^{t(-|x|^\beta\Delta)}$ is defined to derive
explicit solution formula of the $k$-th model. Consequently, a
detailed analysis of the kernal function is given followed by the
space-time estimates of the admissible triplets. Section 4 is
devoted to the radial solution of the nonlinear case $F(u)=|u|^bu$.
The well-posedness results of local solution and the small global
solution are given by the contraction mapping technique.
\section{The linear $k$-th model and its integral solution}
In this paper, we always denote:
$$\lambda=\lambda(n)=\frac{n-2}2,\ \mu(k)=\frac{n-2}2+k,\ \mbox{and}\ \mu(\beta,k)=\frac{2\mu(k)}{2-\beta}.$$
Here $n\geq2$ stands for the dimension of Euclidean space and $k$
stands for the degree of  spherical harmonic subspace. For
simplicity, we denote $\mu=\mu(\beta,k)$. We start this section from
recalling the the spherical harmonics expansion. Let $$x=r\theta\ \
\ \mbox{and}\ \ \ \xi=\rho\omega \ \ \ \mbox{with}\ \ \
\theta,\omega\in\mathbb{S}^{n-1}.$$ For any $g\in
L^2(\mathbb{R}^n)$, we have
$$g(x)=g(r\theta)=\sum^\infty_{k=0}\sum^{d(k)}_{l=1}a_{k,l}(r)Y_{k,l}(\theta),$$
where $$ \{Y_{k,1},\cdots Y_{k,d(k)}\}$$ is the orthogonal basis of
the space of spherical harmonics of degree $k$ on
$\mathbb{S}^{n-1}$, called $\mathcal{H}^k$, having dimension
$$d(k)=\frac{2k+n-2}{k}C^{k-1}_{n+k-3}\backsimeq<k>^{n-2}.$$
We remark that for n=2, the dimension of $\mathcal{H}^k$ is
independent of $k$. Obviously, we have the orthogonal decomposition
$$L^2(\mathbb{S}^{n-1})=\bigoplus_{k=0}^\infty\mathcal{H}^k.$$
By orthogonality, it gives
$$\|g(x)\|_{L^2_\theta(\mathbb{S}^{n-1})}=\|a_{k,l}(r)\|_{l_{k,l}^2}=\left(\sum_{k=0}^{\infty}\sum_{l=1}^{d(k)}a_{k,l}^2(r)\right)^\frac12.$$

Now we consider the following semilinear spacially weighted
dissipative equation in polar coordinates,
\begin{equation}\label{*}\left\{\begin{array}{ll}
\partial_t u-|x|^{\beta} \Delta u=f(x,t)\ \ (x,t)\in \mathbb{R}^n\times [0,\infty),\\
u(x,0)=u_0(x)\ \ \ \ \ x \in \mathbb{R}^n.
\end{array}\right.\end{equation}
Let $V(t,r,\theta)=u(t,r\theta)$, the initial data
$V_0(r,\theta)=u_0(r\theta)$ and the inhomogeneous term
$F(t,r,\theta)=f(t,r\theta)$. Then $V(t,r,\theta)$ satisfies:
\begin{equation*}\left\{\begin{array}{ll}
\partial_tV-r^\beta(\partial_{rr}V+\frac{n-1}r\partial_rV+\frac1{r^2}\Delta_\theta V)=F(t,r\theta),\\
V(0,r,\theta)=V_0(r\theta).
\end{array}\right.
\end{equation*}
Furthermore, let the initial data $V_0$ and inhomogeneous term $F$
as superposition of spherical harmonic functions, i.e.,
  $$V_0(r\theta)=\sum^\infty_{k=0}\sum^{d(k)}_{l=1}a_{k,l}(r)Y_{k,l}(\theta)\ \ \mbox{and}\ \  F(t,r\theta)=\sum^\infty_{k=0}\sum^{d(k)}_{l=1}b_{k,l}(t,r)Y_{k,l}(\theta).$$
Using the separation of variables, we can write $V(t,r,\theta)$ as a
linear combination of products of radial functions and spherical
harmonics,
$$V(t,r,\theta)=\sum_{k=0}^{\infty}\sum_{l=1}^{d(k)}v_{k,l}(t,r)Y_{k,l}(\theta),$$
where $v_{k,l}$ is given by
\begin{equation*}\left\{\begin{array}{ll}
\partial_tv_{k,l}-r^\beta(\partial_{rr}v_{k,l}+\frac{n-1}r\partial_rv_{k,l}-\frac{k(k+n-2)}{r^2}v_{k,l})=b_{k,l}(t,r),\\
v_{k,l}(0,r)=a_{k,l}(r),
\end{array}\right.
\end{equation*}
for each $k,l\in \mathbb{N}$ and $1\leq l\leq d(k)$. If we denote
the operator
$$A_{\mu(k)}:=-\partial^2_r-\frac{n-1}{r}\partial_r+\frac{\mu^2(k)-\lambda^2(n)}{r^{2}},$$
then, we can rewrite the above equation by the definition of
$A_{\mu(k)}$ as
\begin{equation}\label{kth}\left\{\begin{array}{ll}\partial_t v_{k,l}+r^\beta A_{\mu(k)}v_{k,l}=b_{k,l}(t,r),\\
v_{k,l}(0,r)=a_{k,l}(r).
\end{array}\right.\end{equation}
We call equation \eqref{kth} the $k$-th model. In the rest of this
section, we skip $k$ and $l$ in the notation for convenience's sake
by remembering $\mu=\mu(\beta,k)=\frac{2\mu(k)}{2-\beta}$ and
$\mu(k)=k+\lambda$.


Next, we introduce the generalized $\beta$-Hankel transform and give
the mild solution of the $k$-th model \eqref{kth}.
\begin{defn}
Let $\beta\in[0,2)$, $\phi(r)$ and $\psi(r)$ be integrable functions
in $\mathbb{R}^+$, we define the generalized $\beta$-Hankel
transform of $\phi(r)$ as follows,
$$\mathcal{H}_{\mu }\phi(\rho):=\int_0^{\infty}U(r\rho)\phi(r)r^{n-1}dr,\ U(w)=w^{\frac{2-n-2\beta}{2}}J_{\mu}(\frac2{2-\beta}w^{\frac{2-\beta}{2}})$$
and its inversion on $\psi(\rho)$,
$$\mathcal{H}^{-1}_{\mu }\psi(r):=\int_0^{\infty}V(r\rho)\psi(\rho)\rho^{n-1}d\rho,\ V(w)=w^{\frac{2-n}{2}}J_{\mu }(\frac2{2-\beta}w^{\frac{2-\beta}{2}}),$$
where $J_\mu(x)$ is the first kind of Bessel function of real order
$\mu=\mu(\beta,k)>-\frac12$ defined as
$$J_{\mu}(r):=\frac{(r/2)^\mu}{\Gamma(\mu+1/2)\pi^{1/2}}\int_{-1}^1e^{irt}(1-t^2)^{\mu-1/2}dt.$$
\end{defn}

Before going further, we need to prove $\mathcal{H}^{-1}_{\mu}$ in
Definition 2.1 is the true inverse of $\mathcal{H}_{\mu}$. Given
$\phi\in L(\mathbb{R}^+)$, we have
\begin{eqnarray*}
\mathcal{H}^{-1}_{\mu }\mathcal{H}_{\mu }\phi(r)&=&\int_0^{\infty}V(r\rho)\int_0^{\infty}U(s\rho)\phi(s)s^{n-1}ds\rho^{n-1}d\rho\\
&=&\int_0^\infty\phi(s)s^{n-1}ds\int_0^\infty
V(r\rho)U(s\rho)\rho^{n-1}d\rho.
\end{eqnarray*}
After a proper scaling calculation, we find
\begin{eqnarray*}
&&\int_0^\infty V(r\rho)U(s\rho)\rho^{n-1}d\rho\\
&&=\frac2{2-\beta}r^{\frac{2-n}{2}}s^{\frac{2-n-2\beta}2}\int_0^\infty
J_\mu(\frac2{2-\beta}r^{\frac{2-\beta}2}\rho^{\frac{2-\beta}2})
J_\mu(\frac2{2-\beta}s^{\frac{2-\beta}2}\rho^{\frac{2-\beta}2})
\rho^{\frac{2-\beta}2}d\rho^{\frac{2-\beta}2}\\
&&=r^{\frac{\beta-n}{2}}s^{\frac{2-n-2\beta}2}\delta(\frac2{2-\beta}r^{\frac{2-\beta}2}-\frac2{2-\beta}s^{\frac{2-\beta}2}),
\end{eqnarray*}
where $\delta$ is the delta function. Thus,
\begin{eqnarray*}
\mathcal{H}^{-1}_{\mu }\mathcal{H}_{\mu }\phi(r)&=&\int_0^\infty\phi(s)r^{\frac{\beta-n}{2}}s^{\frac{n-\beta}2}\delta(\frac2{2-\beta}r^{\frac{2-\beta}2}-\frac2{2-\beta}s^{\frac{2-\beta}2})d(\frac2{2-\beta}s^{\frac{2-\beta}2})\\
&=&\phi(r).
\end{eqnarray*}
As a result, $\beta$-Hankel transform and its inverse are
well-defined. We have the following properties for the
$\beta$-Hankel transform:
\begin{pro}
Let $\mathcal{H}_{\mu(\beta,k)}$ and $A_{\mu(k)}$ be defined as
above, then,
\begin{flalign}
&\displaystyle (i)\ \mathcal{H}_{\mu}\ \mbox{and}\ \mathcal{H}_{\mu}^{-1}\ \mbox{are self-adjoint, i.e.,}\ \mathcal{H}_{\mu}=\mathcal{H}_{\mu}^*\ \mbox{and}\ \mathcal{H}_{\mu}^{-1}=\mathcal{H}_{\mu}^{-1*}\nonumber\\
&\displaystyle (ii)\ \int_0^\infty\mathcal{H}_{\mu}^2\phi(\rho)\rho^{\beta+n-1}d\rho=\int_0^\infty \phi^2(r)r^{-\beta+n-1}dr.&\nonumber\\
&\displaystyle (iii)\ \mathcal{H}_{\mu(\beta,k)}(r^\beta
A_{\mu(k)}\phi)(\rho)=\rho^{2-\beta}\mathcal{H}_{\mu(\beta,k)}(\phi)(\rho).&
\nonumber
\end{flalign}
\end{pro}
\begin{proof}
(i) This is obvious from definition.\\
(ii) Observe that
$$\mathcal{H}_\mu(r^{\beta}\phi(r))(\rho)=\rho^{-\beta}\mathcal{H}_\mu^{-1}(\phi(r))(\rho),$$
by combining property (i), one has
\begin{eqnarray*}
&&<\mathcal{H}_\mu\phi(\rho),\rho^{\beta}\mathcal{H}_\mu\psi(\rho)>=<\phi(r),\mathcal{H}_\mu(\rho^{\beta}\mathcal{H}_\mu\psi(\rho))(r)>\\
&&=<\phi(r),r^{-\beta}\mathcal{H}_\mu^{-1}(\mathcal{H}_\mu\psi(\rho))(r)>=<\phi(r),r^{-\beta}\psi(r)>.
\end{eqnarray*}
(iii) Using Definition 2.1 and integrating by parts, we have
\begin{eqnarray}\mathcal{H}_{\mu(\beta,k)}(r^\beta A_{\mu(k)}\phi)(\rho)&=&\int_0^\infty (-\partial^2_r-\frac{n-1}r\partial_r+\frac{\mu(k)^2-\lambda^2}{r^2})\phi(r)U(r\rho)r^{\beta+n-1}dr\nonumber\\
&=&\int_0^\infty r^{\beta+n-3}\bigg\{\left(-\beta(\beta+n-2)+\mu(k)^2-\lambda^2\right)U(r\rho)\nonumber\\
&&+(-2\beta-n+1)rsU^\prime(r\rho)-(r\rho)^2U^{\prime\prime}(r\rho)\bigg\}\phi(r)dr.\label{braces}
\end{eqnarray}
It is evident
$U(r\rho)=(r\rho)^{\frac{2-n-2\beta}2}J_{\mu(\beta,k)}(\frac2{2-\beta}(r\rho)^{\frac{2-\beta}2})$,
and it satisfies the following Bessel equation \cite{Watson}
$$(r\rho)^2U^{\prime\prime}(r\rho)+(n+2\beta-1)rsU^\prime(r\rho)+\big\{(r\rho)^{2-\beta}
+(1-\beta-\frac{n}2)^2-(\frac{2-\beta}2\mu(\beta,k))^2\big\}U(r\rho)=0.$$
Recalling $\mu(\beta,k)=\frac{2}{2-\beta}\mu(k)$,
 we find that the terms in braces of \eqref{braces} is equal to
 $(r\rho)^{2-\beta}U(r\rho)$.
Thus,
$$\mathcal {H}_{\mu(\beta,k)}(r^\beta A_{\mu(k)}\phi)=\int_0^\infty
r^{\beta+n-3}(r\rho)^{2-\beta}
U(r\rho)\phi(r)dr=\rho^{2-\beta}\mathcal {H}_\mu\phi(\rho).$$
\end{proof}
Further we introduce the $\beta$-Hankel convolution operator
$\sharp$.
\begin{defn}
Let $\alpha=\beta-k$, $U$ and $V$ be defined as above. We define the
Delsarte's kernel:
$$D(x,y,z):=\int_0^\infty \eta^\alpha V(x\eta)U(y\eta)U(z\eta)\eta^{n-1}d\eta,$$
the Hankel translate function:
$$f^{*}(x,y):=\int_0^\infty f(z)D(x,y,z)z^{n-1}dz,$$
and the $\beta$-Hankel convolution operator $\sharp$:
$$f\sharp g(x):= \int^\infty_0f^*(x,y)g(y)y^{n-1}dy.$$
\end{defn}
From Definition 2.3, we easily get
\begin{eqnarray*}
f\sharp g(x)&= &\int^\infty_0f^*(x,y)g(y)y^{n-1}dy\\
&=&\int^\infty_0 g(y)y^{n-1}dy\int^\infty_0f(z)D(x,y,z)z^{n-1}dz\\
&=&\mathcal{H}^{-1}(\eta^\alpha\mathcal{H}g(\eta)\mathcal{H}f(\eta))(x),
\end{eqnarray*}
which implies
\begin{equation}\mathcal{H}(f\sharp g)(\eta)=\eta^\alpha \mathcal{H}g(\eta)\mathcal{H}f(\eta).\end{equation}
We summary the properties about Delsarte's kernel $D(x,y,z)$ in the
follow proposition.
\begin{pro}
The following identity holds for $D(x,y,z)$ defined above:
$$D(x,y,z)=\frac{(xyz)^{-\lambda-\beta}x^\beta}{1-\beta/2}\cdot [(\frac2{2-\beta})^3(xyz)^{\frac{2-\beta}2}]^{-\mu}\cdot\frac{2^{\mu-1}\bigtriangleup^{2\mu-1}}{\Gamma(\mu+\frac12)\Gamma(\frac12)}$$
where $\bigtriangleup$ is the area of a triangle with sides
$(\frac{2(x)^\frac{2-\beta}2}{2-\beta},\frac{2(y)^\frac{2-\beta}2}{2-\beta},\frac{2(z)^\frac{2-\beta}2}{2-\beta})$
if such a triangle exits or zero otherwise. Besides, we have
\begin{eqnarray}
\int^\infty_0x^{k-\beta}D(x,y,z)x^{n-1}dx&=&\Gamma(\mu+1)^{-1}(2-\beta)^{-\mu}(yz)^{k-\beta}\label{x};\\
\int^\infty_0y^{k}D(x,y,z)y^{n-1}dy&=&\Gamma(\mu+1)^{-1}(2-\beta)^{-\mu}x^{k}z^{k-\beta}\label{y};\\
\int^\infty_0z^{k}D(x,y,z)z^{n-1}dz&=&\Gamma(\mu+1)^{-1}(2-\beta)^{-\mu}x^{k}y^{k-\beta}\label{z}.
\end{eqnarray}
\end{pro}
\begin{proof}
Writing $D(x,y,z)$ in integration, we easily find
\begin{eqnarray*}
&&D(x,y,z)=\int_0^\infty \eta^{\beta-k}V(x\eta)U(y\eta)U(z\eta)\eta^{n-1}d\eta\\
&&=\frac{(xyz)^{-\lambda-\beta}x^\beta}{1-\beta/2}\int_0^\infty
J_{\mu}(\frac{2(x\eta)^\frac{2-\beta}2}{2-\beta})
J_{\mu}(\frac{2(y\eta)^\frac{2-\beta}2}{2-\beta})J_{\mu}(\frac{2(z\eta)^\frac{2-\beta}2}{2-\beta})\frac{d\eta^{\frac{2-\beta}2}}{\eta^{(1-\beta/2)(\mu-1)}}.\\
\end{eqnarray*}
Thus, the first identity follows from \cite{Watson}. The proof of
the next three integration identities are basically the same, so we
only prove $\eqref{x}$. According to Definition 2.3, we immediately
have:
\begin{equation}\int^\infty_0U(xs)D(x,y,z)x^{n-1}dx=s^{\alpha}U(ys)U(zs).\label{ii1}\end{equation}
By the asymptotic behavior of Bessel function,
$$J_{\nu-\frac12}(x)\sim\frac1{\Gamma(\nu+\frac12)}(\frac x2)^{\nu-\frac12},$$
we get
\begin{equation*}
\lim_{s\rightarrow0}{s^\alpha}{
U(xs)}=\lim_{s\rightarrow0}s^{\alpha}
(xs)^{\frac{2-n-2\beta}2}J_{\mu}(\frac2{2-\beta}(xs)^{\frac{2-\beta}2})
=\Gamma(\mu+1)^{-1}(2-\beta)^{-\mu}x^{k-\beta}.
\end{equation*}
Similarly, one will also find
$$\lim_{s\rightarrow0}{s^{-k}}{ V(ys)}=\Gamma(\mu+1)^{-1}(2-\beta)^{-\mu}y^{k}.$$
Multiplying $s^{\alpha}$ to \eqref{ii1} and let $s$ go to 0 on both
sides, we derive \eqref{x}.
\end{proof}
With this proposition, we can prove the next lemma which is the
Young's inequality for $\beta$-Hankel convolution.
\begin{lem}
For the convolution $\sharp$ defined above, we have:
$$\left(\int_0^\infty|\frac{f\sharp g(x)}{x^k}|^ax^{2k+n-1-\beta}dx\right)^{\frac1a}$$
$$\leq\left|\Gamma(\mu+1)^{-1}(2-\beta)^{-\mu}\right| \left(\int_0^\infty|\frac{f(z)}{z^k}|^{b}z^{2k+n-1-\beta}dz\right)^{\frac1b}\left(\int_0^\infty|\frac{g(y)}{y^k}|^cy^{2k+n-1-\beta}dy\right)^{\frac1c},$$
where $1+\frac1a=\frac1b+\frac1c$.
\end{lem}
\begin{proof}

We start from the integration of translate function $f^*(x,y)$.
\begin{eqnarray*}
&&\int_0^\infty|\frac{f^*(x,y)}{y^{k-\beta}}|^py^{2k+n-1-\beta}dy=\int_0^\infty|f^*(x,y)|^py^{2k+n-1-\beta-p(k-\beta)}dy\\
&=&\int_0^\infty\left|\int_0^\infty\frac{f(z)}{z^k}(xy)^ky^{-\beta}\frac{D(x,y,z)}{(xy)^ky^{-\beta}}z^{k+n-1}dz\right|^py^{2k+n-1-\beta-p(k-\beta)}dy.
\end{eqnarray*}
From \eqref{z} we konw
$$\int^\infty_0z^{k}D(x,y,z)z^{n-1}dz=\Gamma(\mu+1)^{-1}(2-\beta)^{-\mu}x^{k}y^{k-\beta}.$$
By applying Jensen's inequality on has
$$\left|\int_0^\infty\frac{f(z)}{z^k}(xy)^ky^{-\beta}\frac{D(x,y,z)}{(xy)^ky^{-\beta}}z^{k+n-1}dz\right|^p$$
$$\leq(\Gamma(\mu+1)^{-1}(2-\beta)^{-\mu})^{p-1}\int_0^\infty\left|\frac{f(z)}{z^k}(xy)^ky^{-\beta}\right|^p\frac{D(x,y,z)}{(xy)^ky^{-\beta}}z^{k+n-1}dz.$$

By changing the order of integration again
\begin{eqnarray*}
&&\int_0^\infty|\frac{f^*(x,y)}{y^{k-\beta}}|^py^{2k+n-1-\beta}dy\leq(\Gamma(\mu+1)^{-1}(2-\beta)^{-\mu})^{p-1}\\
&&\cdot\int_0^\infty\left|\frac{f(z)}{z^k}\right|^pd\sigma(z)\int_0^\infty[(xy)^ky^{-\beta}]^{p-1}D(x,y,z)z^{k+n-1}y^{2k+n-1-\beta-p(k-\beta)}dy\\
&=&(\Gamma(\mu+1)^{-1}(2-\beta)^{-\mu})^{p-1}\int_0^\infty\left|\frac{f(z)}{z^k}\right|^pd\sigma(z)x^{k(p-1)}\int_0^\infty D(x,y,z)d\sigma(y)\\
&=&(\Gamma(\mu+1)^{-1}(2-\beta)^{-\mu})^{p}x^{kp}\int_0^\infty\left|\frac{f(z)}{z^k}\right|^pz^{2k+n-1-\beta}dz.\\
\end{eqnarray*}
Similarly, one can obtain
$$\int_0^\infty|\frac{f^*(x,y)}{x^{k}}|^px^{2k+n-1-\beta}dx
\leq(\Gamma(\mu+1)^{-1}(2-\beta)^{-\mu})^{p}y^{(k-\beta)p}\int_0^\infty\left|\frac{f(z)}{z^k}\right|^pz^{2k+n-1-\beta}dz.$$
Let $d\eta(x)=x^{2k+n-1-\beta}dx$. By Young's inequality
(\cite{Titchmarsh}), we have
\begin{eqnarray*}
|f\sharp g|&=&|\int_0^\infty f^*(x,y)g(y)y^{n-1}dy|=|\int_0^\infty\frac{f^*(x,y)}{y^{k-\beta}}\cdot\frac{g(y)}{y^k}\cdot y^{2k+n-1-\beta}dy|\\
&\leq& \left(\int_0^\infty|\frac{f^*(x,y)}{y^{k-\beta}}|^p\cdot|\frac{g(y)}{y^k}|^{q}d\eta(y)\right)^{\frac1m}\\
&&\cdot\left(\int_0^\infty|\frac{f^*(x,y)}{y^{k-\beta}}|^{p}d\eta(y)\right)^{1-\frac1q}\cdot\left(\int_0^\infty|\frac{g(y)}{y^k}|^{q}d\eta(y)\right)^{1-\frac1p}\\
&=&I\cdot II \cdot III.
\end{eqnarray*}
And we have
$$II^{\frac
q{q-1}}\leq(\Gamma(\mu+1)^{-1}(2-\beta)^{-\mu})^{p}x^{kp}\int_0^\infty\left|\frac{f(z)}{z^k}\right|^pd\eta(z)$$
and
\begin{eqnarray*}
&&\int_0^\infty |\frac{f\sharp g}{x^k}|^md\eta(x)\leq\int_0^\infty (\frac{I\cdot II \cdot III}{x^k})^md\eta(x)\\
&&=\left(\int_0^\infty|\frac{g(y)}{y^k}|^{q}d\eta(y)\right)^{\frac {m(p-1)}p} \int_0^\infty \left(\int_0^\infty|\frac{f^*(x,y)}{y^{k-\beta}}|^{p}\cdot|\frac{g(y)}{y^k}|^{q}d\eta(y)\right)\cdot II^m x^{-km}d\eta(x) \\
&&\leq(\Gamma(\mu+1)^{-1}(2-\beta)^{-\mu})^{\frac{mp(q-1)}{q}}\left(\int_0^\infty|\frac{g(y)}{y^k}|^{q}d\eta(y)\right)^{\frac {(p-1)m}p} \left(\int_0^\infty|\frac{f(z)}{z^k}|^{p}d\eta(z)\right)^{\frac {m(q-1)}q}\\
&&\cdot\int_0^\infty \left(\int_0^\infty|\frac{f^*(x,y)}{y^{k-\beta}}|^{p}\cdot|\frac{g(y)}{y^k}|^{q}d\eta(y)\right)\cdot x^{k\frac {mp(q-1)}{q}} x^{-km}d\eta(x) \\
&&=(\Gamma(\mu+1)^{-1}(2-\beta)^{-\mu})^{\frac{pm(q-1)}{q}}\left(\int_0^\infty|\frac{g(y)}{y^k}|^{q}d\eta(y)\right)^{\frac {(p-1)m}p} \left(\int_0^\infty|\frac{f(z)}{z^k}|^{p}d\eta(z)\right)^{\frac {(q-1)m}q}\\
&&\cdot\int_0^\infty|\frac{g(y)}{y^k}|^{q}y^{(\beta-k)p}d\eta(y)\cdot \int_0^\infty|\frac{f^*(x,y)}{x^{k}}|^{p}d\eta(x)\\
&&\leq(\Gamma(\mu+1)^{-1}(2-\beta)^{-\mu})^{\frac{mp(q-1)}{q}}\left(\int_0^\infty|\frac{g(y)}{y^k}|^{q}d\eta(y)\right)^{\frac {m(p-1)}p} \left(\int_0^\infty|\frac{f(z)}{z^k}|^{p}d\eta(z)\right)^{\frac {m(q-1)}q}\\
&&\cdot\int_0^\infty|\frac{g(y)}{y^k}|^{q}d\eta(y)\cdot (\Gamma(\mu+1)^{-1}(2-\beta)^{-\mu})^{p}\int_0^\infty\left|\frac{f(z)}{z^k}\right|^pd\eta(z)\\
&&=(\Gamma(\mu+1)^{-1}(2-\beta)^{-\mu})^{m}\left(\int_0^\infty|\frac{g(y)}{y^k}|^{q}d\eta(y)\right)^{\frac {m}q} \left(\int_0^\infty|\frac{f(z)}{z^k}|^{p}d\eta(z)\right)^{\frac {m}p}.\\
\end{eqnarray*}
That is,
$$\left(\int_0^\infty |\frac{f\sharp g}{x^k}|^md\eta(x)\right)^{\frac1m}\leq(\Gamma(\mu+1)^{-1}(2-\beta)^{-\mu})\left(\int_0^\infty|\frac{g(y)}{y^k}|^{q}d\eta(y)\right)^{\frac {1}q} \left(\int_0^\infty|\frac{f(z)}{z^k}|^{p}d\eta(z)\right)^{\frac {1}p}.$$
\end{proof}
Finally, we can derive the integral solution of linear $k$-th model
\eqref{kth}. Applying the $\beta$-Hankel transform, we get
\begin{equation*}\label{**}
\left\{\begin{array}{ll}\partial_t \mathcal{H}_{\mu}v +\rho^{2-\beta}\mathcal{H}_{\mu}v =\mathcal{H}_{\mu}b, \\
\mathcal{H}_{\mu}v (0,\rho)=(\mathcal{H}_{\mu}a )(\rho).
\end{array}\right.\end{equation*}
Solving the ODE and further applying $\mathcal{H}_{\mu}^{-1}$, we
get its explicit solution formula which can be also represented in
terms of Hankel convolution:
\begin{eqnarray*}
v (t,r)&=&\mathcal{H}_{\mu}^{-1}[\exp(-\rho^{2-\beta}t)\mathcal{H}_{\mu}a ](r)+\mathcal{H}_{\mu}^{-1}[\int_0^t\exp(-\rho^{2-\beta}(t-\tau))\mathcal{H}_{\mu}b(\rho,\tau)d\tau ](r)\\
&=&\mathcal{H}_{\mu}^{-1}[\frac{\exp(-\rho^{2-\beta}t)}{\rho^{\beta-k}}](r)\sharp
a
(r)+\int_0^t\mathcal{H}_{\mu}^{-1}[\frac{\exp(-\rho^{2-\beta}(t-\tau))}{\rho^{\beta-k}}](r)\sharp
b(r,\tau)d\tau.
\end{eqnarray*}
Define the solution's kernal $K_{\mu}(r,t)$  by
$$K_{\mu}(r,t):=\mathcal{H}_{\mu}^{-1}[\frac{\exp(-\rho^{2-\beta}t)}{\rho^{\beta-k}}](r)$$
and the solution semigroup $S_{\mu}(t)(\triangleq e^{r^\beta
A_{\mu(k)}t})$ by
$$ S_{\mu}(t)f:=K_{\mu}(r,t) \sharp f,$$
then, the solution can be written in a simple form
\begin{equation}\label{mild}
v(t,r)=S_{\mu}(t)a(r)+\int_0^tS_{\mu}(t-\tau)b(\tau,r)d\tau.
\end{equation}

\section{Space-time estimates for the linear $k$-th model}
In this section, we analyse the kernal $K_{\mu}(r,t)$ and the
semigroup $S_{\mu}(t)$. After that we discuss the space-time
estimates of solution to the $k$-th model.

We start from the definition of $K_{\mu}(r,t)$.
\begin{eqnarray}
K_{\mu}(r,t)&=&\mathcal{H}^{-1}_{\mu}(\frac{\exp(-\rho^{2-\beta}t)}{\rho^{\beta-k}})(r)\nonumber\\
&=&\int_0^\infty(r\rho)^{-\lambda}J_{\mu}(\frac2{2-\beta}(r\rho)^{\frac{2-\beta}2})
\frac{\exp(-\rho^{2-\beta}t)}{\rho^{\beta-k}}\rho^{n-1}d\rho\nonumber\\
&=&\frac{2r^{-\lambda}}{2-\beta}\int_0^\infty\exp(-\rho^{2-\beta}t)J_{\mu}(\frac{2r^{\frac{2-\beta}{2}}}{2-\beta}\rho^{\frac{2-\beta}{2}})
(\rho^\frac{2-\beta}2)^{\mu+1}d\rho^{\frac{2-\beta}2}\nonumber\\
&=&\{(2-\beta)t\}^{-\mu-1}\exp(-\frac{r^{2-\beta}}{(2-\beta)^2t})r^k.\label{kernal}
\end{eqnarray}
The last equality is due to the identity from \cite{Watson}
\begin{equation}\label{Wat}\int_0^\infty J_{\nu}(at)\exp(-p^2t^2)t^{\nu+1}dt=\frac{a^\nu}{(2p^2)^{\nu+1}}\exp(-\frac{a^2}{4p^2}).\end{equation}
For the solution semigroup $S_{\mu}(t)$, by changing the order of
integration, we get
\begin{equation}\label{homo}S_{\mu}(t)a(r)
=\int_0^\infty\widetilde{ K}(\rho,r,t)a(\rho)\rho^{n-1}d\rho,
\end{equation}
where $\displaystyle\widetilde{K}(\rho,r,t):=$
$$\frac2{2-\beta}r^{-\lambda}\rho^{-\lambda-\beta}\int_0^\infty J_{\mu}(\frac2{2-\beta}(r\xi)^{\frac{2-\beta}2})
J_{\mu}(\frac2{2-\beta}(\rho\xi)^{\frac{2-\beta}2})\exp(-\xi^{2-\beta}t)\xi^{\frac{2-\beta}2}d\xi^{\frac{2-\beta}2}.
$$
This integral is equivalent to the Weber's second exponential
integral after a proper scaling calculation(See \cite{Watson} p395).
The convergence is secured by $\exp(-\xi^{2-\beta}t)$ and $\mu>0$.
Moreover, we have
$$\widetilde{K}(r,\rho,t)=\frac{r^{-\lambda}\rho^{-\lambda-\beta}}{(2-\beta)t}\cdot\exp\{-\frac{1}{t(2-\beta)^2}(r^{2-\beta}+\rho^{2-\beta})\}\cdot I_{\mu}(\frac{2(r\rho)^{\frac{2-\beta}2}}{(2-\beta)^2t}),$$
where $$I_{\mu}(x)=i^{-\mu}J_{\mu}(ix)$$ stands for the modified
Bessel function. Consequently,
\begin{eqnarray*}
S_{\mu}(t)a(r)&=&\frac{2r^{-\lambda}\exp(-\frac{r^{2-\beta}}{t(2-\beta)^2})}{(2-\beta)^2t}\cdot\\
&&\int_0^\infty
i^{-\mu}J_\mu(\frac{2ir^{\frac{2-\beta}2}}{(2-\beta)^2t}\rho^{\frac{2-\beta}2})\exp(-\frac{\rho^{2-\beta}}{t(2-\beta)^2})a(\rho)
\rho^{\frac{2-\beta}2(\mu+1)-k}d\rho^{\frac{2-\beta}2}.
\end{eqnarray*}

If we apply the H\"{o}lder's inequality directly and recall the
identity \eqref{Wat} again, we obtain:
\begin{eqnarray*}
&&|\frac{S_{\mu}(t)a(r)}{r^k}|\leq\|\frac{a(\rho)}{\rho^k}\|_{L^\infty}
\frac{2r^{-\lambda-k}\exp(-\frac{r^{2-\beta}}{t(2-\beta)^2})}{(2-\beta)^2t}\cdot\\
&&\left|\int_0^\infty
i^{-\mu}J_\mu(\frac{2ir^{\frac{2-\beta}2}}{(2-\beta)^2t}\rho^{\frac{2-\beta}2})\exp(-\frac{\rho^{2-\beta}}{t(2-\beta)^2})
\rho^{\frac{2-\beta}2(\mu+1)}d\rho^{\frac{2-\beta}2}\right|.
\end{eqnarray*}
That is
\begin{equation}\label{infty}
\|\frac{S_{\mu}(t)a(r)}{r^k}\|_{L^\infty(r)}\leq\|\frac{a(\rho)}{\rho^k}\|_{L^\infty(\rho)}.
\end{equation}
In fact, by applying Young's inequality in Lemma 2.5, we have the
following $L^p-L^q$ estimates for homogeneous part of the solution.
\begin{lem} Let $d\eta(r)=r^{2k+n-1-\beta}dr$ and $1\leq q\leq p\leq \infty$. Then $S_{\mu}(t)a(r)$ satisfies the following
estimates,
\begin{equation}\label{young}
\|\frac{S_{\mu}(t)a(r)}{r^k}\|_{L^p_{d\eta(r)}} \leq
C(\beta,\mu,p,q)
t^{\frac{2k+n-\beta}{2-\beta}(\frac1p-\frac1q)}\|\frac{
a(r)}{r^k}\|_{L^q_{d\eta(r)}},
\end{equation}
where constant $C(\beta,\mu,p,q)$ is independent of $k$.
\end{lem}
\begin{proof}
Since $S_{\mu}(t)a(r)= K_{\mu}(r,t) \sharp a(r)$ and by recalling
\eqref{kernal}
we find that,
\begin{eqnarray*}
\|\frac{K_{\mu}(r,t)}{r^k}\|_{L^m_{d\eta(r)}}&=&\{(2-\beta)t\}^{-\mu-1}(\int_0^\infty\exp(-\frac{mr^{2-\beta}}{(2-\beta)^2t})r^{2k+n-1-\beta}dr)^\frac1m\\
&\leq&
(2-\beta)^{\frac{4k+2n-\beta-2}{(2-\beta)m}-\mu-1}m^{-\frac{2k+n-\beta}{(2-\beta)m}}\Gamma(\mu+1)^{\frac1m}
t^{\frac{2k+n-\beta}{2-\beta}(\frac1m-1)}.
\end{eqnarray*}
Hence, the Young's inequality gives the desired inequality
\eqref{young} by taking $1+\frac1p=\frac1m+\frac1q$ and
\begin{equation}\label{constant}
C(\beta,\mu,p,q)=[(2-\beta)^{(2\mu+1)}\Gamma(\mu+1)]^{\frac1p-\frac1q}m^{-\frac{2k+n-\beta}{(2-\beta)m}}.
\end{equation}
Since $\frac1p-\frac1q\leq0$ and $(2-\beta)^{(2\mu+1)}\Gamma(\mu+1)$
goes to infinity as $k\rightarrow\infty$, we find that the constant
$C(\beta,\mu,p,q)$ is independent of $k$.
\end{proof}
It is remarkable that \eqref{young} generalizes the result of
\eqref{infty} by taking $m=1$, $p=q$,
$$\|\frac{S_{\mu}(t)a(r)}{r^k}\|_{L^p_{d\eta(r)}}
\leq \|\frac{ a(r)}{r^k}\|_{L^p_{d\eta(r)}}.$$ At the moment we
state the space-time estimates for the homogeneous part of solution
$v$ given in \eqref{mild}. Its proof can be made by following
\cite{Giga} (see also \cite{Miao08}).
\begin{lem}
(i) Let $\psi$ satisfy
$\|\frac{\psi}{r^k}\|_{L^q_{d\eta(r)}}<\infty$ and $(m,p,q)$ be any
admissible triplet. Then, $\frac{S_{\mu}(t)\psi}{r^k}\in
L^m(I;L_{d\eta(r)}^p(\mathbb{R}^+))\cap
C_b(I;L_{d\eta(r)}^q(\mathbb{R}^+))$ with the estimate
\begin{equation}\label{s-t1}\|\frac{S_{\mu}(t)\psi}{r^k}\|_{L^m(I;L_{d\eta(r)}^p)}\leq C\|\frac{\psi}{r^k}\|_{L^q_{d\eta(r)}}\end{equation}
for $0<T\leq\infty$, where $C$ is a positive constant independent of $k$.\\
(ii) Let $\psi$ satisfy
$\|\frac{\psi}{r^k}\|_{L^q_{d\eta(r)}}<\infty$ and $(m,p,q)$ be any
generalized admissible triplet. Then, $\frac{S_{\mu}(t)\psi}{r^k}\in
\mathcal{C}_m(I;L_{d\eta(r)}^p(\mathbb{R}^+))\cap
C_b(I;L_{d\eta(r)}^q(\mathbb{R}^+))$ with the estimate
\begin{equation}\label{s-t2}\|\frac{S_{\mu}(t)\psi}{r^k}\|_{\mathcal{C}_m(I;L_{d\eta(r)}^p)}\leq C\|\frac{\psi}{r^k}\|_{L^q_{d\eta(r)}}\end{equation}
for $0<T\leq\infty$, where $C$ is a positive constant independent of $k$.\\
Hereafter, for a Banach space $X$, we denote by $C_b(I;X)$ the space
of bounded continuous functions from $I$ to $X$.
\end{lem}
\begin{proof}The statement (ii) follows easily from Lemma 3.1. It suffices to prove (i).
For the case $p=q$ and $m=\infty$, the space-time estimate is true
from \eqref{young}. We now consider the case $p>q$. Assume
$(\tilde{m},\tilde{p},\tilde{q})$ be an admissible triplet and
define the operator
$$U\psi=\|\frac{S_{\mu}(t)\psi}{r^k}\|_{L^{\tilde{p}}_{d\eta(r)}}$$
from an weighted $L^q$ space to functions on $[0,T)$. As the Young's
inequality \eqref{young} gives
$$U\psi\leq Ct^{\frac{-n+\beta-2k}{2-\beta}(\frac1{\tilde{q}}-\frac1{\tilde{p}})}\|\frac{\psi}{r^k}\|_{L^{\tilde{q}}_{d\eta(r)}}= C t^{-\frac1{\tilde{m}}}\|\frac{\psi}{r^k}\|_{L^{\tilde{q}}_{d\eta(r)}}.$$
It is easy to see that
\begin{eqnarray*}
m(t:|U\psi|>\tau)&\leq&m\{t:Ct^{-\frac1{\tilde{m}}}\|\frac{\psi}{r^k}\|_{L_{d\eta(r)}^{\tilde{q}}}>\tau\}\\
&=&m\{t:t<(\frac{C\|\frac{\psi}{r^k}\|_{L_{d\eta(r)}^{\tilde{q}}}}{\tau})^{\tilde{m}}\}\\
&\leq&(\frac{C\|\frac{\psi}{r^k}\|_{L_{d\eta(r)}^{\tilde{q}}}}{\tau})^{\tilde{m}},
\end{eqnarray*}
which implies that $U$ is a weak type $(\tilde{q},\tilde{m})$
operator. On the other hand, $U$ is sub-additive and satisfies that:
$$U\psi=\|\frac{S_{\mu}(t)\psi}{r^k}\|_{L^{\tilde{p}}_{d\eta(r)}}\leq C\|\frac{\psi}{r^k}\|_{L_{d\eta(r)}^{\tilde{p}}}$$
for $q\leq \tilde{p}\leq \infty$, which means that $U$ is a
$(\tilde{p},\infty)$ type operator. For any given admissible triplet
$(m,p,q)$, we choose proper $(\tilde{m},\tilde{p},\tilde{q})$ and
$\theta$ such that
$$\frac1{q}=\frac{\theta}{\tilde{q}}+\frac{1-\theta}{\tilde{p}},$$
$$\frac1{m}=\frac{\theta}{\tilde{m}}+\frac{1-\theta}{\infty},$$
$$p=\tilde{p}.$$
Then, the operator $U$ is of type $(q,m)$ by the Marcinkiewicz
interpolation theorem, i.e.,
$$\|U\psi\|_{L^{m}}\leq C \|\frac{\psi}{r^k}\|_{L_{d\eta(r)}^{q}},$$
which is just the desired result \eqref{s-t1}.
\end{proof}

\begin{rem}Let $m=\infty$ and $p=q=2$, from \eqref{s-t1} we derive
$$\forall \ k\geq0 \ \ \ \ \sup_{t>0}\int_0^\infty|S_{\mu}(t)\psi_k|^2r^{n-1-\beta}dr\leq C\int_0^\infty|\psi_k|^2r^{n-1-\beta}dr.$$
Summing over with $k$ and $l$, we obtain:
$$\sup_{t>0}\int\sum_{k,l}|S_{\mu}(t)\psi_k|^2r^{n-1-\beta}dr\leq C\int_0^\infty\sum_{k,l}|\psi_k|^2r^{n-1-\beta}dr.$$
That is
$$\|v(t,r,\theta)\|_{L^\infty_tL^2({r^{n-1-\beta}dr})L^2_{\theta}}\leq C\|u_0\|_{L^2({r^{n-1-\beta}dr})L^2_{\theta}}.$$
\end{rem}
Now, we move to the nonhomogeneous part of solution. From here and
following, we denote $$\mathbb{G}(f)(t,r):=
\int_0^tS_{\mu}(t-\tau)f(\tau,r)d\tau)$$ and
$$\gamma=\frac{n-\beta+2k}{2-\beta}.$$ As matter of fact, we have the following space-time
estimates in $\mathcal{L}$ space framework.
\begin{lem}
For $b>0$ and $T>0$, let 
$q_0=b\gamma$, $I=[0,T)$. Assume $q\geq q_0>1$ and $(m,p,q)$ is an
admissible triplet
satisfying $p>b+1$.\\
(i) If $\frac{f}{r^k}\in L^{\frac m{b+1}}(I;L_{d\eta(r)}^{\frac
p{b+1}})$, then,
$$\|\frac{\mathbb{G}f}{r^k}\|_{L^\infty(I;L_{d\eta(r)}^q)}\leq CT^{1-\frac{b\gamma}q}\|\frac{f}{r^k}\|_{L^{\frac m{b+1}}(I;L_{d\eta(r)}^{\frac p{b+1}})}$$
for $p\leq q(1+b)$ and
$$\|\frac{\mathbb{G}f}{r^k}\|_{L^\infty(I;L_{d\eta(r)}^q)}\leq CT^{1-\frac{b\gamma}q}\||\frac{f}{r^k}|^{\frac1{b+1}}\|^{\theta(b+1)}_{L^\infty(I;L^q_{d\eta(r)})}\||\frac{f}{r^k}|^{\frac1{b+1}}\|_{L^{m}(I;L_{d\eta(r)}^{p})}^{(1-\theta)(b+1)}$$
for $p> q(1+b)$, where $\theta=\frac{p-q(b+1)}{(b+1)(p-q)}$.\\
(ii) If $\frac{f}{r^k}\in L^{\frac m{b+1}}(I;L_{d\eta(r)}^{\frac
p{b+1}})$, then,
$$\|\frac{\mathbb{G}f}{r^k}\|_{L^m(I;L_{d\eta(r)}^p)}\leq CT^{1-\frac{b\gamma}q}\|\frac{f}{r^k}\|_{L^{\frac m{b+1}}(I;L_{d\eta(r)}^{\frac p{b+1}})}$$
for $p\leq q(1+b)$ and
$$\|\frac{\mathbb{G}f}{r^k}\|_{L^m(I;L_{d\eta(r)}^p)}\leq CT^{1-\frac{b\gamma}q}\||\frac{f}{r^k}|^{\frac1{b+1}}\|^{\theta(b+1)}_{L^\infty(I;L^q_{d\eta(r)})}\||\frac{f}{r^k}|^{\frac1{b+1}}\|_{L^{m}(I;L_{d\eta(r)}^{p})}^{(1-\theta)(b+1)}$$
for $p> q(1+b)$, where $\theta$ is the same as in (i).
\end{lem}
\begin{proof}
First we prove (i). Consider the case when $p\leq q(b+1)$. Using
Lemma 3.1 and H\"{o}lder's inequality on $t$, one has
\begin{eqnarray*}\|\frac{\mathbb{G}f}{r^k}\|_{L^\infty(I;L_{d\eta(r)}^q)}&\leq& C\int_0^t(t-\tau)^{-\gamma(\frac{b+1}p-\frac1q)}\|\frac{f(\tau,r)}{r^k}\|_{L_{d\eta(r)}^{\frac p{b+1}}}d\tau\\
&\leq& C \left(\int_0^t(t-\tau)^{-\gamma(\frac{b+1}p-\frac1q)\chi} d\tau\right)^{\frac1\chi}\|\frac{f(\tau,r)}{r^k}\|_{L^{\frac m{b+1}}(I;L_{d\eta(r)}^{\frac p{b+1}})}\\
&\leq& CT^{1-\frac{b\gamma}{q}}\|\frac{f}{r^k}\|_{L^{\frac
m{b+1}}(I;L_{d\eta(r)}^{\frac p{b+1}})},
\end{eqnarray*}
where $\frac1\chi=1-\frac{b+1}{m}$ and $C=C(\mu,p,q,b)$. For the
case $p> q(b+1)$, by means of the Riesz interpolation theorem and
H\"{o}lder's inequality, we have:
\begin{eqnarray*}
\|\frac{\mathbb{G}f}{r^k}\|_{L^\infty(I;L_{d\eta(r)}^q)}&\leq& \int_0^t\left\||\frac{f}{r^k}|^{\frac1{b+1}}\right\|_{L^{q(b+1)}_{d\eta(r)}}^{b+1}d\tau\\
&\leq& C\int_0^t\left\Vert|\frac{f(\tau,r)}{r^k}|^{\frac1{b+1}}\right\Vert^{(b+1)\theta}_{L^q_{d\eta(r)}}\left\||\frac{f(\tau,r)}{r^k}|^{\frac1{b+1}}\right\|_{L^p_{d\eta(r)}}^{(b+1)(1-\theta)}d\tau\\
&\leq&CT^{1-\frac{(b+1)((1-\theta)}{m}}\left\||\frac{f(\tau,r)}{r^k}|^{\frac1{b+1}}\right\|^{(b+1)\theta}_{C(I;L^q_{d\eta(r)})}
\left\||\frac{f(\tau,r)}{r^k}|^{\frac1{b+1}}\right\|_{L^m(I;L^p_{d\eta(r)})}^{(b+1)(1-\theta)}\\
&=&
CT^{1-\frac{b\gamma}{q}}\left\||\frac{f(\tau,r)}{r^k}|^{\frac1{b+1}}\right\|^{(b+1)\theta}_{C(I;L^q_{d\eta(r)})}
\left\||\frac{f(\tau,r)}{r^k}|^{\frac1{b+1}}\right\|_{L^m(I;L^p_{d\eta(r)})}^{(b+1)(1-\theta)},
\end{eqnarray*}
where $\theta$ satisfies
$$\frac1{q(b+1)}=\frac{\theta}q+\frac{1-\theta}p,$$ and
$$1=\frac{(1+b)(1-\theta)}{m}+\frac1\chi.$$ We now prove (ii). For
the case $p\leq q(b+1)$, by Lemma 3.1 and Young's inequality on $t$,
one has
\begin{eqnarray*}\|\frac{\mathbb{G}f}{r^k}\|_{L^m(I;L_{d\eta(r)}^p)}&\leq& C\left\| \int_0^t(t-\tau)^{-\gamma(\frac{b+1}p-\frac1p)}\|\frac{f(\tau,r)}{r^k}\|_{L_{d\eta(r)}^{\frac p{b+1}}}d\tau\right\|_{L^m}\\
&\leq& C \left(\int_0^T\tau^{-\frac{b\gamma}p\chi} d\tau\right)^{\frac1\chi}\|\frac{f(\tau,r)}{r^k}\|_{L^{\frac m{b+1}}(I;L_{d\eta(r)}^{\frac p{b+1}})}\\
&\leq& CT^{1-\frac{b\gamma}{q}}\|\frac{f}{r^k}\|_{L^{\frac
m{b+1}}(I;L_{d\eta(r)}^{\frac p{b+1}})},
\end{eqnarray*}
where $1+\frac1m=\frac1\chi+\frac{b+1}{m}$. For the case $p>q(b+1)$,
by a similar manner as the proof of (i), one has
\begin{eqnarray*}
\|\frac{\mathbb{G}f}{r^k}\|_{L^m(I;L_{d\eta(r)}^p)}&\leq& C\left\|\int_0^t(t-\tau)^{-\gamma(\frac1q-\frac1p)}\||\frac{f}{r^k}|^{\frac1{b+1}}\|_{L^{q(b+1)}_{d\eta(r)}}^{b+1}d\tau\right\|_{L^m}\\
&\leq& C\left\|\int_0^t(t-\tau)^{-\gamma(\frac1q-\frac1p)}\left\Vert|\frac{f(\tau,r)}{r^k}|^{\frac1{b+1}}\right\Vert^{(b+1)\theta}_{L^q_{d\eta(r)}}\left\||\frac{f(\tau,r)}{r^k}|^{\frac1{b+1}}\right\|_{L^p_{d\eta(r)}}^{(b+1)(1-\theta)}d\tau\right\|_{L^m}\\
&\leq&C\left(\int_0^T\tau^{-\gamma(\frac1q-\frac1p)\chi}d\tau\right)
^{\frac1\chi}\left\||\frac{f(\tau,r)}{r^k}|^{\frac1{b+1}}\right\|^{(b+1)\theta}_{C(I;L^q_{d\eta(r)})}
\left\||\frac{f(\tau,r)}{r^k}|^{\frac1{b+1}}\right\|_{L^m(I;L^p_{d\eta(r)})}^{(b+1)(1-\theta)}\\
&\leq&
CT^{1-\frac{b\gamma}{q}}\left\||\frac{f(\tau,r)}{r^k}|^{\frac1{b+1}}\right\|^{(b+1)\theta}_{C(I;L^q_{d\eta(r)})}
\left\||\frac{f(\tau,r)}{r^k}|^{\frac1{b+1}}\right\|_{L^m(I;L^p_{d\eta(r)})}^{(b+1)(1-\theta)},
\end{eqnarray*}
where $\theta$ and $\chi$ satisfy
$$\frac1{q(b+1)}=\frac\theta{q}+\frac{1-\theta}p,\ \ \ \ \ 1+\frac1m=\frac{(b+1)(1-\theta)}{m}+\frac1\chi$$
with $q<q(1+b)<p$.
\end{proof}
In fact, concerning the nonhomogeneous part, Lemma 3.4 has its
counterpart in the the $\mathcal{C}$ space framework. The estimates
can be proved by following \cite{Miao08}. We state the result here
and leave the proof to readers.
\begin{lem} For $b>0$ and $T>0$, let $\gamma=\frac{n-\beta+2k}{2-\beta}$, $q_0=b\gamma$, $I=[0,T)$. Assume $q\geq q_0>1$ and $(m,p,q)$ is an admissible triplet
satisfying $p>b+1$.\\
(i) If $\frac{f}{r^k}\in \mathcal{C}_{\frac
m{b+1}}(I;L_{d\eta(r)}^{\frac p{b+1}})$, then,
$$\|\frac{\mathbb{G}f}{r^k}\|_{L^\infty(I;L_{d\eta(r)}^q)}\leq CT^{1-\frac{b\gamma}q}\|\frac{f}{r^k}\|_{\mathcal{C}_{\frac m{b+1}}(I;L_{d\eta(r)}^{\frac p{b+1}})}$$
for $p\leq q(1+b)$ and
$$\|\frac{\mathbb{G}f}{r^k}\|_{L^\infty(I;L_{d\eta(r)}^q)}\leq CT^{1-\frac{b\gamma}q}\||\frac{f}{r^k}|^{\frac1{b+1}}\|^{\theta(b+1)}_{L^\infty(I;L^q_{d\eta(r)})}\||\frac{f}{r^k}|^{\frac1{b+1}}\|_{\mathcal{C}_{m}(I;L_{d\eta(r)}^{p})}^{(1-\theta)(b+1)}$$
for $p> q(1+b)$, where $\theta=\frac{p-q(b+1)}{(b+1)(p-q)}$.\\
(ii) If $\frac{f}{r^k}\in \mathcal{C}_{\frac
m{b+1}}(I;L_{d\eta(r)}^{\frac p{b+1}})$, then,
$$\|\frac{\mathbb{G}f}{r^k}\|_{\mathcal{C}_m(I;L_{d\eta(r)}^p)}\leq CT^{1-\frac{b\gamma}q}\|\frac{f}{r^k}\|_{\mathcal{C}_{\frac m{b+1}}(I;L_{d\eta(r)}^{\frac p{b+1}})}$$
for $p\leq q(1+b)$ and
$$\|\frac{\mathbb{G}f}{r^k}\|_{\mathcal{C}_m(I;L_{d\eta(r)}^p)}\leq CT^{1-\frac{b\gamma}q}\||\frac{f}{r^k}|^{\frac1{b+1}}\|^{\theta(b+1)}_{L^\infty(I;L^q_{d\eta(r)})}\||\frac{f}{r^k}|^{\frac1{b+1}}\|_{\mathcal{C}_{m}(I;L_{d\eta(r)}^{p})}^{(1-\theta)(b+1)}$$
for $p> q(1+b)$, where $\theta$ is the same as in (i).
\end{lem}

\section{Proof of Theorem 1.4}
Following the similar procedure as the linear $k$-th model, we also
obtain the corresponding integral equation of \eqref{radial}:
\begin{eqnarray}\label{nonlinear-mild}
u(t,r)&=&\mathcal{T}(u):=S_{\mu}(t)u_0(r)+\int_0^tS_{\mu}(t-\tau)F(u(\tau,r))d\tau\\
&=&S_{\mu}u_0(t,r)+\mathbb{G}(F(u(\tau,r))).\nonumber
\end{eqnarray}
We call the solution of the integral form \eqref{nonlinear-mild} the
mild solution.
\begin{eqnarray*}
\|u(t,r)\|_{X(I)}&\leq &\|S_{\mu}u_0(t,r)\|_{X(I)}+\|\mathbb{G}(F(u(\tau,r))\|_{X(I)}\\
&=&I+II.
\end{eqnarray*}
Indeed, we have
\begin{equation}\label{I}
I\leq
\|S_{\mu}(t)u_0(r)\|_{L^m(I;L_{d\eta}^p(\mathbb{R}^+))}+\|S_{\mu}(t)u_0(r)\|_{L^\infty(I;L_{d\eta}^q(\mathbb{R}^+))}\leq
C_1 \|u_0\|_{L^q_{d\eta}}
\end{equation}
and
\begin{eqnarray}\label{II}
II&=&\|\int_0^tS_{\mu}(t-\tau)|u|^bud\tau\|_{L^m(I;L_{d\eta}^p(\mathbb{R}^+))}+\|\int_0^tS_{\mu}(t-\tau)|u|^bud\tau\|_{L^\infty(I;L_{d\eta}^q(\mathbb{R}^+))}\nonumber\\
&\leq&\left\{\begin{array}{ll}CT^{1-\frac{b\gamma}q}\||u|^bu\|_{L^{\frac{m}{b+1}}(I;L^{\frac{p}{b+1}}_{d\eta})}, \ \ \ \mbox{for}\ \ b+1<p\leq q(b+1),\\
 CT^{1-\frac{b\gamma}q}\|u\|^{\theta(b+1)}_{L^\infty(I;L^q_{d\eta})}\|u\|^{(1-\theta)(b+1)}_{L^m(I;L^p_{d\eta})},\ \ \ \ \ \ \mbox{for}\ \ p>q(b+1),\ \ \ \end{array}\right.\nonumber\\
 &\leq&C_2T^{1-\frac{b\gamma}q}\|u\|_{X(I)}^{b+1}.
\end{eqnarray}
Combining the estimates for $I$ and $II$, we derive
\begin{equation}\label{Tu}
\|\mathcal{T}u\|_{X(I)}\leq C_1 \|u_0\|_{L^q_{d\eta}}+
C_2T^{1-\frac{b\gamma}q}\|u\|_{X(I)}^{b+1}.
\end{equation}
Moreover,
\begin{eqnarray*}
&&\|T(u-v)\|_{X(I)}=\|\int_0^t S_\mu(t-\tau)(|u|^bu-|v|^bv)d\tau\|_{X(I)}\\
&&\leq\|\int_0^t S_\mu(t-\tau)\left((|u|^b+|v|^b)|u-v|\right)d\tau\|_{X(I)}\\
&&\leq\|\int_0^t S_\mu(t-\tau)(|u|^b|u-v|)d\tau\|_{X(I)}+\|\int_0^t
S_\mu(t-\tau)(|v|^b|u-v|)d\tau\|_{X(I)}.
\end{eqnarray*}
At the same time, we obtain
\begin{eqnarray*}
&&\|\int_0^t S_\mu(t-\tau)(|u|^b|u-v|)d\tau\|_{X(I)}\\
&&\leq\left\{\begin{array}{ll}CT^{1-\frac{b\gamma}q}\||u|^b(u-v)\|_{L^{\frac{m}{b+1}}(I;L^{\frac{p}{b+1}}_{d\eta})}, \ \ \ \mbox{for}\ \ b+1<p\leq q(b+1),\\
 CT^{1-\frac{b\gamma}q}\||u|^\frac{b}{b+1}(u-v)^{\frac1{b+1}}\|^{\theta(b+1)}_{L^\infty(I;L^q_{d\eta})}\||u|^\frac{b}{b+1}(u-v)^{\frac1{b+1}}\|^{(1-\theta)(b+1)}_{L^m(I;L^p_{d\eta})},\ \ \ \ \ \ \mbox{for}\ \ p>q(b+1),\ \ \
\end{array}\right.\\
&&\leq\left\{\begin{array}{ll}CT^{1-\frac{b\gamma}q}\|u^b\|_{L^{\frac{m}{b}}(I;L^{\frac{p}{b}}_{d\eta})}\|(u-v)\|_{L^{m}(I;L^{p}_{d\eta})}, \ \ \ \mbox{for}\ \ b+1<p\leq q(b+1),\\
 CT^{1-\frac{b\gamma}q}\|u\|_{L^\infty(I;L^{q}_{d\eta})}^{\theta b}\|u-v\|^{\theta}_{L^\infty(I;L^{q}_{d\eta})}
 \|u\|_{L^m(I;L^{p}_{d\eta})}^{(1-\theta)b}
 \|u-v\|^{(1-\theta)}_{L^m(I;L^{p}_{d\eta})},\ \ \ \ \ \  \mbox{for}\ \ p>q(b+1),\ \ \
\end{array}\right.\\
 &&\leq C_2T^{1-\frac{b\gamma}q}\|u\|_{X(I)}^{b}\|u-v\|_{X(I)},
\end{eqnarray*}
and
$$\|\int_0^t S_\mu(t-\tau)(|v|^b|u-v|)d\tau\|_{X(I)}\leq C_2T^{1-\frac{b\gamma}q}\|v\|_{X(I)}^{b}\|u-v\|_{X(I)}.$$
Thus, we have
\begin{equation}\label{T(u-v)}
\|\mathcal{T}(u-v)\|_{X(I)}\leq
C_2T^{1-\frac{b\gamma}q}(\|u\|_{X(I)}^{b}+\|v\|_{X(I)}^{b})\|u-v\|_{X(I)}.
\end{equation}
Now we define the metric space as follows,
$$X_p^{Sol}(I)=\left\{u\in X(I)\bigg|\|u\|_{X(I)}\leq 2C_1\|u_0\|_{L^q_{d\eta}},\ (2C_1)^bC_2T^{1-\frac{b\gamma}q}\|u_0\|_{L^q_{d\eta}}^{b}\leq\frac1{2}\right\}.$$
The estimates of \eqref{Tu} and \eqref{T(u-v)} imply that
$\mathcal{T}u$ is a contraction mapping from $X_p^{Sol}$ to itself.
We obtained the results (i) and (ii) by applying the Banach
contraction mapping principle. Concerning (iii), by standard
argument, one is able to show
$$\lim_{t\rightarrow T^*}\|u(t)\|_{L^q_{d\eta}}=\infty.$$
Meanwhile, for arbitrary $t<s<T^*$ with
$\|u(t)\|_{L^q_{d\eta}}<\infty$, by following a similar procedure as
above, one can find the unique solution in
$$X_p^{Sol}([t,s])=\left\{u\in X([t,s])\bigg|\|u\|_{X([t,s])}\leq 2C_1\|u(t)\|_{L^q_{d\eta}},\ (2C_1)^bC_2|s-t|^{1-\frac{b\gamma}q}\|u(t)\|_{L^q_{d\eta}}^{b}\leq\frac1{2}\right\}.$$
Thus, there exits $\varepsilon_0>0$ such that
$$\varepsilon_0\leq(2C_1)^bC_2|s-t|^{1-\frac{b\gamma}q}\|u(t)\|_{L^q_{d\eta}}^{b}\leq\frac1{2},$$
which gives
$$\|u(t)\|_{L^q_{d\eta}}\geq\frac{C\varepsilon_0}{(T^*-t)^{\frac1b-\frac{\gamma}{q}}}.$$

\newpage

\begin{thebibliography}{00}
\bibitem{FP03} N. Burq, F. Planchon, J. G. Stalker and A. S. Tahvildar-Zadeh, {\it Strichartz estimates for the wave and Schr\"{o}dinger
equations with the inverse-square potential}, J. Funct. Anal.,
203(2)(2003), 519-549.
\bibitem{BR97} J. J. Betancor and L. Rodr\'{\i}guez-Mesa, {\it On Hankel Convolutors on Certain Hankel Transformable Function Spaces}, Glasgow Math. J., 39(1997), 351-369.
\bibitem{CHEN} I. K. Chen, {\it Spherical averaged endpoint Strichartz estimates for the two-dimensional Schr$\ddot{o}$dinger equation with the inverse square
potential}, J. Hyperbolic Differ. Equations, 7(3)(2010), 365-382.
2010.
\bibitem{Giga} Y. Giga, {\it Solutions for Semilinear Parabolic Equations in $L^p$ and Regularity of Weak Solutions of the Navier-Stokes System}, J. Differential Equations, 61(1986), 186-212.
\bibitem{Haimo} D. T. Haimo, {\it Integral equations associated with Hankel convolutions}, Trans. Amer. Math. Soc., 116(1965), 330-375.
\bibitem{Hirschman} I. I. Hirschman, {\it Variation diminishing Hankel transforms}, J. Analyse Math., 8(1960), 307-336.
\bibitem{MB95} I. Marrero and J. J. Betancor, {\it Hankel convolution of generalized functions, Rendiconti di Matematica},  VII(15)(1995), 351-380.
\bibitem{Miao01} C. Miao, {\it Time-Space estimates of solutions to general Semilinear Parabolic equations}, Tokyo J. Math., 24(1)(2001).
\bibitem{Miao08} C. Miao, B. Yuan and B. Zhang, {\it Well-posedness of the Cauchy problem for the fractional power dissipative equations}, Nonlinear Analysis, 68(2008), 461-484.
\bibitem{Miao15} C. Miao, J. Zhang and J. Zheng, {\it Maximal estimate for Schr$\ddot{o}$dinger equations with inverse-square potential}, Pacific J. Mathematics, 273(1)(2015), 1-20.
\bibitem{PP97} R. S. Pathak and P. K. Pandey, {\it Sobolev Type Spaces Associaited with Bessel Operators}, J. Math. Anal. Appl., 215(1997), 95¨C111.
\bibitem{FP01} F. Planchon, J. G. Stalker and A. S. Tahvildar-Zadeh, {\it $L^\beta$ estimates for the wave equation with the
inverse-square potential}, Discrete Contin. Dynam. Systems,
9(2)(2003), 427-442.
\bibitem{Tao} T. Tao, {\it Spherically averaged endpoint Strichartz estimates for the two-dimensional
Schr$\ddot{o}$dinger equation}, Commun. Partial Differential
Equations, 25(7¨C8)(2000), 1471-1485.
\bibitem{Titchmarsh} E. C. Titchmarsh, ``Introduction to the Theory of Fourier Integrals", 2nd Ed., Oxford Univ. Press, London, 1984.
\bibitem{Watson} G. N. Watson, ``A Treatise on the theory of Bessel Functions", 2nd Ed., Cambridge, 1966.
\bibitem{Zemanian66} A. H. Zemanian, {\it A distributional Hankel transformation}, SIAM J. Appl. Math., 14(1966), 561-576.
\bibitem{Zemanian67} A. H. Zemanian, {\it The Hankel transformation of certain distribution of rapid growth}, SIAM J. Appl. Math., 14(1966), 678-690.
\bibitem{Zhai09} Z. Zhai, {\it Strichartz type estimates for fractional heat equations}, J. Math. Anal. Appl., 356(2009), 642-658.
\end{thebibliography}






\end{document}